\documentclass[10pt]{amsart}
\usepackage{amssymb}

\newtheorem{proposition}{Proposition}[section]
\newtheorem{lemma}[proposition]{Lemma}
\newtheorem{corollary}[proposition]{Corollary}
\newtheorem{theorem}[proposition]{Theorem}
\newtheorem{remark}[proposition]{Remark}

\theoremstyle{definition}

\newcommand{\selabel}[1]{\label{se:#1}}

\def\<{\leq}
\def\>{\geq}

\def\e{\varepsilon}

\def\ot{\otimes}
\def\ra{\rightarrow}

\date{}

\begin{document}
\title{The Green rings of Taft algebras}
\author{Huixiang Chen}
\address{School of Mathematical Science, Yangzhou University,
Yangzhou 225002, China}
\email{hxchen@yzu.edu.cn}
\author{Fred Van Oystaeyen}
\address{Department of Mathematics and Computer
Science, University of Antwerp, Middelheimlaan 1, B-2020 Antwerp,
Belgium}
\email{fred.vanoystaeyen@ua.ac.be}
\author{Yinhuo Zhang}
\address{Department WNI, University of Hasselt, Universitaire Campus, 3590 Diepeenbeek,Belgium }
\email{yinhuo.zhang@uhasselt.be}
\begin{abstract}
We compute the Green ring of the Taft algebra $H_n(q)$,
where $n$ is a positive integer greater than 1,
and $q$ is an $n$-th root of unity.
It turns out that the Green ring
$r(H_n(q))$ of the Taft algebra $H_n(q)$ is a commutative ring generated
by two elements subject to certain relations defined recursively. Concrete examples
for $n=2,3, ... , 8$ are given.
\end{abstract}

\maketitle
\section*{\bf Introduction}\selabel{1}

The tensor product of representations of a Hopf algebra is an important ingredient
in the representation theory of Hopf algebras and quantum groups. In particular the decomposition of the tensor product
of indecomposable modules into a direct sum of indecomposables has received enormous
attention. For modules over a group algebra this information is encoded in the
structure of the Green ring (or the representation ring) for finite groups,
 \cite{Archer, BenCar, BenPar, BrJoh, Green, HTW}).
For modules over a Hopf algebra or a quantum group there are results
by Cibils on a quiver quantum group \cite{Cib},
by Witherspoon on  the quantum double of a finite group
\cite{With}, by Gunnlaugsd$\acute{\rm o}$ttir on the half quantum groups
(or Taft algebras) \cite{Gunn}, and by Chin on
the coordinate Hopf algebra of quantum ${\rm SL}(2)$ at a root of unity
\cite{Chin}. However, the Green rings of those Hopf algebras are either equal to
the Grothendick rings (in the semisimple cases) or not yet computed because of the complexity.

In this paper, we compute the Green rings of
Taft algebras. It turns out that the Green ring of a Taft algebra is
much more complicated than its Grothendick ring.
In Section 1, we recall some basic definitions and results,
and make some preparations for the rest of the paper.
In Section 2, we recall the indecomposable modules over the Taft algebra $H_n(q)$ from \cite{Cib, Gunn},
using the terminology of matrix representations, where
$n$ is a positive integer $\geqslant 2$, and $q$ is a primitive $n$-th root of
unity in the ground field $k$. There are $n^2$ non-isomorphic
finite dimensional indecomposable modules over $H_n(q)$, and all of them are
uniserial. Moreover, for each $1\leqslant l\leqslant n$, there are exactly $n$
finite dimensional indecomposable $H_n(q)$-modules $M(l, r)$, $r\in{\mathbb Z}_n$,
up to isomorphism. Every indecomposable projective $H_n(q)$-module
is $n$-dimensional. In Section 3,
we describe the Green ring of Taft algebra $H_n(q)$.
We first recall the decomposition formula of the tensor product of
two indecomposable modules over $H_n(q)$ from \cite{Cib, Gunn}.
From the decomposition formula, we know that the tensor product of any two
$H_n(q)$-modules is commutative. Moreover, the tensor product of two
indecomposable non-projective modules has a simple summand if and only if the two indecomposable
modules have the same dimension. Finally we describe the structure
of the Green ring $r(H_n(q))$ of $H_n(q)$. We show  that the Green ring $r(H_n(q))$
is generated by two elements
subject to certain relations which can be defined recursively.

\section{\bf Preliminaries}\selabel{2}

Throughout, we work over a fixed field $k$. Unless
otherwise stated, all algebras, Hopf algebras and modules are
defined over $k$; all modules are left modules and finite dimensional;
all maps are $k$-linear; dim,
$\otimes$ and Hom stand for $\mbox{\rm dim}_k$, $\otimes_k$ and
Hom$_k$, respectively. For the theory of Hopf algebras and quantum groups, we
refer to \cite{Ka, Maj, Mon, Sw}.

Let $0\not=q\in k$. For any integer $n>0$, set
$(n)_q=1+q+\cdots +q^{n-1}$.
Observe that $(n)_q=n$ when $q=1$, and
$$
(n)_q=\frac{q^n-1}{q-1}
$$
when $q\not= 1$.
Define the $q$-factorial of $n$ by
$(0)!_q=1$ and
$(n)!_q=(n)_q(n-1)_q\cdots (1)_q$ for $n>0$.
Note that $(n)!_q=n!$ when $q=1$, and
$$
(n)!_q=
\frac{(q^n-1)(q^{n-1}-1)\cdots (q-1)}{(q-1)^n}
$$
when $n>0$ and $q\not= 1$.
The  $q$-binomial coefficients
$
\left(\begin{array}{c}
n\\
i\\
\end{array}\right)_q
$
are defined inductively as follows for $0\leqslant i\leqslant n$:
$$
\left(\begin{array}{c}
n\\
0\\
\end{array}\right)_q
=1=
\left(\begin{array}{c}
n\\
n\\
\end{array}\right)_q
\quad\quad
\mbox{ for } n\geqslant 0,$$
$$
\left(\begin{array}{c}
n\\
i\\
\end{array}\right)_q
=
q^i
\left(\begin{array}{c}
n-1\\
i\\
\end{array}\right)_q
+
\left(\begin{array}{c}
n-1\\
i-1\\
\end{array}\right)_q
\quad \quad
\mbox{ for } 0< i< n.$$
It is well-known that
$
\left(\begin{array}{c}
n\\
i\\
\end{array}\right)_q$
is a polynomial in $q$ with integer coefficients and with value at $q=1$
is equal to the usual binomial coefficient
$
\left(\begin{array}{c}
n\\
i\\
\end{array}\right)$, and that
$$
\left(\begin{array}{c}
n\\
i\\
\end{array}\right)_q
=\frac{(n)!_q}{(i)!_q(n-i)!_q}
$$
when
$(n-1)!_q\not = 0$ and $0<i<n$
(see \cite[Page 74]{Ka}).

Throughout this paper, we fix an integer
$n\geqslant 2$ and assume that the field $k$ contains an
$n$-th primitive root $q$ of unity.
Then we have
$$
\left(\begin{array}{c}
n\\
i\\
\end{array}\right)_q\not=0,\hspace{0.3cm} (i)!_q\not=0,\hspace{0.3cm}
\ \mathrm{for}\ 0<i<n,
$$
and that $n$ is not divisible by the characteristic of $k$, i.e.
$\displaystyle{\frac{1}{n}}\in k$.

The Taft algebra $H_n(q)$ is generated by two elements $g$ and $h$
subject to the relations (see \cite{Ta})
$$g^n=1,\ h^n=0,\ hg=qgh.$$
$H_n(q)$ is a Hopf algebra with coalgebra structure $\Delta$ and antipode $S$ given by
$$\begin{array}{lll}
\Delta(g)=g\ot g,& \Delta(h)=1\ot h+h\ot g,&\e(g)=1,\\
\e(h)=0,&S(g)=g^{-1}=g^{n-1},& S(h)=-q^{-1}g^{n-1}h.\\
\end{array}$$
Note that dim$H_n(q)=n^2$ and $\{g^ih^j|0\leqslant i, j\leqslant n-1\}$
forms a $k$-basis for $H_n(q)$. When $n=2$, $H_2(q)$ is exactly Sweedler's
4-dimensional Hopf algebra.

Let $H$ be a Hopf algebra. The representation ring $r(H)$ and $R(H)$
can be defined as follows. $r(H)$ is the abelian group generated by the
isomorphism classes $[V]$ of finite dimensional $H$-modules $V$
modulo the relations $[M\oplus V]=[M]+[V]$. The multiplication of $r(H)$
is given by the tensor product of $H$-modules, that is,
$[M][V]=[M\ot V]$. Then $r(H)$ is an associative ring.
$R(H)$ is an associative $k$-algebra defined by $k\ot_{\mathbb Z}r(H)$.
Note that $r(H)$ is a free abelian group with a $\mathbb Z$-basis
$\{[V]|V\in{\rm ind}(H)\}$, where ${\rm ind}(H)$ denotes the category
of finite dimensional indecomposable $H$-modules.

\section{\bf Representations of $H_n(q)$}\selabel{3}

For a module $M$ over a finite dimensional algebra $A$,
let ${\rm rl}(M)$ denote the Loewy length (=radical length=socle length)
of $M$, and let ${\rm l}(M)$ denote the length of $M$.
Let $P(M)$ denote the projective cover of $M$, and let $I(M)$
denote the injective hull of $M$.

Cibils constructed an $nd$-dimensional Hopf algebra $kZ_n(q)/I_d$ in \cite{Cib},
where $q$ is an $n$-th root of unity in $k$ with order $d$.
He classified the indecomposable modules over $kZ_n(q)/I_d$, and gave
the decomposition of the tensor products of two arbitrary indecomposable
modules there. When $q$ is a primitive $n$-th root of unity,
$kZ_n(q)/I_n$ is isomorphic to $H_n(q)$ (see \cite{Cib}).
Therefore, from \cite{Cib}, one can get the classification of
indecomposable modules and the decomposition of the tensor product
of two indecomposable modules over $H_n(q)$.
For the completeness, we will describe the indecomposable modules
over $H_n(q)$ in this section, using the terminology of matrix representation.

Let $G(H_n(q))$ denote the group of group-like elements in $H_n(q)$.
Then $G(H_n(q))=\{1, g, \cdots, g^{n-1}\}$ is a cyclic group
of order $n$ generated by $g$. The group algebra $kG(H_n(q))$
is a Hopf subalgebra of $H_n(q)$. There is a Hopf algebra epimorphism
$\pi: H_n(q)\ra kG(H_n(q))$ defined by $\pi(g)=g$ and $\pi(h)=0$.
Since $k$ contains an $n$-th primitive root of unity,
the group algebra $kG(H_n(q))$ is semisimple.
It follows that Ker$\pi=\langle h\rangle\supseteq J(H_n(q))$, the Jacobson radical of $H_n(q)$.
On the other hand, since $H_n(q)h=hH_n(q)$ and $h^n=0$, $J(H_n(q))\supseteq
(h)=H_n(q)h$, the ideal of $H_n(q)$ generated by $h$.
Hence Ker$\pi=(h)=J(H_n(q))$.
Thus, an $H_n(q)$-module $M$ is semisimple if and only if $h\cdot M=0$,
and moreover $M$ is simple
if and only if $h\cdot M=0$ and $M$ is simple as a module
over the Hopf subalgebra $kG(H_n(q))$.
Therefore, we have the following lemma.

\begin{lemma}\label{sim}
There are $n$ non-isomorphic simple $H_n(q)$-modules $S_i$, and
each $S_i$ is 1-dimensional and determined by
$$g\cdot v=q^iv,\ h\cdot v=0,\ v\in S_i,$$
where $i\in{\mathbb Z}_n:={\mathbb Z}/(n)$.  \hfill$\Box$
\end{lemma}

Note that $J(H_n(q))^m=H_n(q)h^m$ for all $m\geqslant 1$.
Hence $J(H_n(q))^{n-1}\neq 0$, but $J(H_n(q))^n=0$.
This means that the Loewy length of $H_n(q)$ is $n$.
Since every simple $H_n(q)$-module is 1-dimensional,
${\rm l}(M)={\rm dim}(M)$ for any $H_n(q)$-module $M$.

Now let $M$ be any $H_n(q)$-module.
Since $J(H_n(q))^s=H_n(q)h^s=h^sH_n(q)$, we have
rad$^s(M)=h^s\cdot M$ for all $s\geqslant 1$.

\begin{lemma}\label{MStruc1}
Let $1\leqslant l\leqslant n$ and $i\in \mathbb{Z}$.
Then there is an algebra map $\rho_{l,i}:H_n(q)\ra M_l(k)$
given by
$$\rho_{l,i}(g)=\left(
                  \begin{array}{ccccc}
                    q^i&&&&\\
                    &q^{i-1}&&&\\
                    &&q^{i-2}&&\\
                    &&&\ddots&\\
                    &&&&q^{i-l+1}\\
                  \end{array}
                \right),\
\rho_{l,i}(h)=\left(
                \begin{array}{ccccc}
                  0 &  &  &  &  \\
                  1 & 0 &  &  &  \\
                    & 1 & \ddots &  & \\
                    &   & \ddots & 0 &  \\
                    &   &   & 1 & 0 \\
                \end{array}
              \right).$$
Let $M(l,i)$ denote the corresponding left $H_n(q)$-module.
\end{lemma}

\begin{proof}
It follows from a straightforward verification.
\end{proof}

There is a $k$-basis $\{v_1, v_2, \cdots, v_l\}$ of $M(l, i)$ such that
$g\cdot v_j=q^{i-j+1}v_j$ for all $1\leqslant j\leqslant l$ and
$$h\cdot v_j=\left\{\begin{array}{ll}
v_{j+1},& 1\leqslant j\leqslant l-1,\\
0,& j=l.\\
\end{array}\right.
$$
Hence we have $v_j=h^{j-1}\cdot v_1$ for all $2\leqslant j\leqslant l$.
Such a basis is called {\it standard basis} of $M(l, i)$.
For any integer $i$, we will often regard $i$ as its image under
the canonical projection ${\mathbb Z}\ra{\mathbb Z}_n:={\mathbb Z}/(n)$.
We have the following lemma.

\begin{lemma}\label{head-soc}
For any $1\leqslant l\leqslant n$ and $i\in\mathbb{Z}$, let $M(l,i)$ be the
$H_n(q)$-module defined as in Lemma \ref{MStruc1}. Then
\begin{enumerate}
\item ${\rm soc}(M(l,i))=kv_l\cong S_{i-l+1}$ and
$M(l,i)/{\rm rad}(M(l,i))\cong S_{i}$.\\
\item $M(l, i)$ is indecomposable and uniserial.\\
\item If $1\leqslant l'\leqslant n$ and $i'\in\mathbb{Z}$, then
$M(l, i)\cong M(l', i')$ if and only if $l'=l$ and $i'=i$
in ${\mathbb Z}_n$.
\end{enumerate}
\end{lemma}

\begin{proof}
(1) Since $J(H_n(q))=(h)=hH_n(q)=H_n(q)h$, soc$(M(l,i))=\{v\in M(l,i)|h\cdot v=0\}
=kv_l$ and rad$(M(l,i))=h\cdot M(l,i)={\rm span}\{v_2, \cdots, v_l\}$.
It follows that soc$(M(l,i))\cong S_{i-l+1}$ and $M(l,i)/{\rm rad}(M(l, i))
\cong S_{i}$.

(2) By (1), soc$(M(l,i))$ is simple, and hence $M(l, i)$ is indecomposable.
Since $h^{l-1}\cdot M(l,i)\neq0$ and $h^l\cdot M(l,i)=0$, rl$(M(l,i))=l$.
Hence $l(M(l,i))={\rm rl}(M(l,i))$, and so $M(l,i)$ is uniserial.

(3) Obvious.
\end{proof}
As a consequence, we obtain the following:
\begin{corollary}\label{pro-inj}
Let $1\leqslant l\leqslant n$ and $i\in{\mathbb Z}_n$. Then\\
\begin{enumerate}
\item$M(l,i)$ is simple if and only if $l=1$. In this case,
$M(1, i)\cong S_i$.\\
\item $M(l,i)$ is projective (injective) if and only if $l=n$.\\
\item $M(n,i)\cong P(S_i)\cong I(S_{i+1})$.
\end{enumerate}
\end{corollary}

\begin{proof}
(1): Follows from Lemma \ref{head-soc}(1).

(2) and (3): Note that any finite dimensional Hopf algebra is a Frobenius algebra,
and hence is a self-injective algebra. If $l=n$, then it follows from
\cite[Lemma 3.5]{Ch3} that $M(n,i)$ is projective and injective.

For any $0\leqslant i\leqslant n-1$, let $e_i=\frac{1}{n}\sum\limits_{j=0}^{n-1}q^{-ij}g^j$.
Then $\{e_0, e_1, \cdots, e_{n-1}\}$ is a set of orthogonal idempotents
such that $\sum\limits_{i=0}^{n-1}e_i=1$. We also have $ge_i=q^ie_i$ and $h^{n-1}e_i\neq 0$.
Therefore, $H_n(q)e_i={\rm span}\{e_i, he_i, \cdots, h^{n-1}e_i\}\cong M(n,i)$.
Thus, we have a decomposition of the regular module $H_n(q)$ as follows:
$$H_n(q)=\bigoplus\limits_{i=0}^{n-1}H_n(q)e_i\cong\bigoplus\limits_{i=0}^{n-1}M(n, i).$$
Hence $M(n,i)\cong P(S_i)$, and $M(n, 0)$, $M(n, 1)$, $\cdots$, $M(n, n-1)$ are
all non-isomorphic indecomposable projective $H_n(q)$-modules. So (2) and (3) follow
from Lemma \ref{head-soc}.
\end{proof}

Since the indecomposable projective $H_n(q)$-modules are uniserial,
any indecomposable $H_n(q)$-module is uniserial and is isomorphic to
a quotient of a indecomposable projective module.
Thus, we have the following theorem (see \cite[Page 467]{Cib}).

\begin{theorem}\label{IMClass}
Up to isomorphism, there are $n^2$ indecomposable finite dimensional
$H_n(q)$-modules as follows
$$\{M(l, i)|1\leqslant l\leqslant n, 0\leqslant i\leqslant n-1\}.$$   \hfill$\Box $
\end{theorem}

\section{\bf The Green Ring of Taft Algebra $H_n(q)$}\selabel{4}

We have already known that there are $n^2$ non-isomorphic
indecomposable modules over $H_n(q)$. They are
$$\{M(l,r)|1\leqslant l\leqslant n, r\in\mathbb{Z}_n\}.$$

The following lemma follows from a straightforward verification.

\begin{lemma}\label{1-l}
Let $1\leqslant l\leqslant n$ and $r, r'\in\mathbb{Z}_n$. Then
$$M(l,r)\ot S_{r'}\cong S_{r'}\ot M(l,r)\cong M(l, r+r')$$
as $H_n(q)$-modules. In particular, $S_r\otimes S_{r'}\cong S_{r+r'}$
and $M(l, r)\cong S_r\ot M(l,0)\cong M(l,0)\ot S_r$. \hfill$\Box$
\end{lemma}

Cibils and Gunnlaugsd${\rm\acute{o}}$ttir derived the decomposition
formulas of the tensor product of two indecomposable modules over
$kZ_n(q)/I_n$ and the half-quantum group $u^+_q$ in \cite{Cib} and \cite{Gunn},
respectively. From \cite[Theorem 4.1]{Cib} or \cite[Theorem 3.1]{Gunn},
one gets the following Propositions \ref{pro}, \ref{cor<n} and \ref{cor>n}.

\begin{proposition}\label{pro}
Let $2\leqslant l\leqslant n$ and $r, r'\in{\mathbb Z}_n$.
Then we have the $H_n(q)$-module isomorphisms
$$M(l,r)\ot M(n, r')\cong M(n, r')\ot M(l, r)
\cong \bigoplus\limits_{i=1}^lM(n, r+r'+i-l).\ \ \ \ \ \ \ \  {\ \ \ \ \ \ \ \ \Box}$$
\end{proposition}

\begin{proposition}\label{cor<n}
Let $1\leqslant l, l'<n$ and $r, r'\in{\mathbb Z}_n$.
If $l+l'\leqslant n$, then
$$M(l,r)\ot M(l',r')\cong\bigoplus\limits_{i=1}^{l_0}
M(|l-l'|-1+2i,r+r'+i-l_0),$$
where $l_0={\rm min}\{l, l'\}$. \ \ \ \ \ \ \ \ \ \ \ \ \ \  \ \ \ \ \ \   $\Box$
\end{proposition}

\begin{proposition}\label{cor>n}
Let $1\leqslant l, l'<n$ and $r, r'\in{\mathbb Z}_n$.
If $l+l'>n$, then
$$M(l,r)\ot M(l',r')\cong(\bigoplus\limits_{i=1}^{n-l_1}
M(|l-l'|-1+2i,r+r'+i-l_0))
\bigoplus(\bigoplus\limits_{i=1}^{l+l'-n}
M(n,r+r'+1-i)),$$
where $l_0={\rm min}\{l, l'\}$ and $l_1={\rm max}\{l, l'\}$. \ \ \ \ \ \ \hfill$\Box$
\end{proposition}
Following Propositions  \ref{cor<n} and \ref{cor>n}, we obtain the following:

\begin{corollary}\label{sim-summand}
Let $1\leqslant l, l'\leqslant n-1$ and $r, r'\in{\mathbb Z}_n$.
Then there is a simple summand in $M(l, r)\ot M(l', r')$ if and only if
$l=l'$. $\hfill\Box$
\end{corollary}

The following property of $M(l,r)$ can be derived from Lemma \ref{1-l}, Propositions \ref{pro},
\ref{cor<n} and \ref{cor>n}.

\begin{corollary}\label{commu}
Let $1\leqslant l, l'\leqslant n$ and $r, r'\in{\mathbb Z}_n$.
Then
$$M(l,r)\ot M(l',r')\cong M(l',r')\ot M(l,r). $$ \hfill$\Box$
\end{corollary}

From Theorem \ref {IMClass} and Corollary \ref{commu}, one can deduce the following known result (see \cite[Page 467]{Cib}).

\begin{corollary}\label{Commu}
For any $H_n(q)$-modules $M$ and $N$, there is an $H_n(q)$-module
isomorphism
$$M\ot N\cong N\ot M. $$ \hfill$\Box$
\end{corollary}

In the sequel, we let $a=[S_{-1}]$ and $x=[M(2,0)]$
in the Green ring $r(H_n(q))$ of $H_n(q)$.
From Corollary \ref{Commu}, we know that $r(H_n(q))$ is a
commutative ring.

\begin{lemma}\label{1}
\begin{enumerate}
\item $a^n=1$ and $[M(l, r)]=a^{n-r}[M(l,0)]$ for all
$2\leqslant l\leqslant n$ and $r\in{\mathbb Z}_n$.
\item If $n>2$, then $[M(l+1, 0)]=x[M(l,0)]-a[M(l-1,0)]$
for all $2\leqslant l\leqslant n-1$.
\item$x[M(n,0)]=(a+1)[M(n,0)]$.
\item $r(H_n(q))$ is generated by $a$ and $x$ as a ring.
\end{enumerate}
\end{lemma}

\begin{proof}
(1): Follows from Lemma \ref{1-l} since $[S_0]$ is the identity of the ring
$r(H_n(q))$.

(2): If $n>2$ and $2\leqslant l\leqslant n-1$, then by Propositions \ref{cor<n} and \ref{cor>n}
and Lemma \ref{1-l}, we have
$$\begin{array}{rcl}
M(2,0)\ot M(l, 0)&\cong& M(l-1,-1)\oplus M(l+1,0)\\
&\cong& S_{-1}\ot M(l-1,0)
\oplus M(l+1,0).\\
\end{array}$$
It follows that $[M(l+1, 0)]=x[M(l,0)]-a[M(l-1,0)]$.

(3): By Proposition \ref{pro} and Lemma \ref{1-l}, we have
$$\begin{array}{rcl}
M(2,0)\ot M(n, 0)&\cong& M(n,-1)\oplus M(n,0)\\
&\cong& S_{-1}\ot M(n,0)
\oplus M(n,0)\\
&\cong&(S_{-1}\oplus S_0)\ot M(n,0).\\
\end{array}$$
It follows that $x[M(n,0)]=(a+1)[M(n,0)]$.

(4): Follows from (1), (2) and (3).
\end{proof}

\begin{corollary}\label{2}
Let $u_1, u_2, \cdots$ be a series of elements of the ring
$r(H_n(q))$ defined recursively by $u_1=1$, $u_2=x$ and
$$u_l=xu_{l-1}-au_{l-2},\ l\geqslant 3.$$
Then $[M(l, 0)]=u_l$ for all $1\leqslant l\leqslant n$ and
$(x-a-1)u_n=0$.
\end{corollary}

\begin{proof}
Follows from Lemma \ref{1}.
\end{proof}


Let ${\mathbb Z}[y,z]$ be the polynomial algebra over $\mathbb Z$ in two variables
$y$ and $z$. We define a {\it generalized Fibonacci polynomial} $f_n(y,z)\in{\mathbb Z}[y,z]$, $n\geqslant 1$, recursively as follows:
$$f_1(y,z)=1, \ f_2(y,z)=z,\ \mathrm{and}\
f_n(y,z)=zf_{n-1}(y,z)-yf_{n-2}(y,z),\ n\geqslant 3.$$
Let $I$ be the ideal of ${\mathbb Z}[y,z]$ generated by polynomials $y^n-1$
and $(z-y-1)f_n(y,z)$.

With the above notations, we have the following main result.

\begin{theorem}
The Green ring $r(H_n(q))$ of $H_n(q)$ is isomorphic to the quotient ring
${\mathbb Z}[y,z]/I$.
\end{theorem}

\begin{proof}
By Lemma \ref{1}(4), $r(H_n(q))$ is generated, as a ring, by $a$ and $x$. Hence there is
a unique ring epimorphism $\phi$ from ${\mathbb Z}[y, z]$ to $r(H_n(q))$ such that
$\phi(y)=a$ and $\phi(z)=x$. Since $a^n=1$ by Lemma \ref{1}(1),
$\phi(y^n-1)=0$. Let $\{u_i\}_{i\geqslant 1}$ be the series of elements of $r(H_n(q))$
given in Corollary \ref{2}. It is easy to see that $\phi(f_1(y,z))=u_1$ and $\phi(f_2(y,z))=u_2$.
Now let $i\geqslant 3$ and assume that $\phi(f_{i-2}(y,z))=u_{i-2}$ and $\phi(f_{i-1}(y,z))=u_{i-1}$.
Then
$$\begin{array}{rcl}
\phi(f_i(y,z))&=&\phi(zf_{i-1}(y,z)-yf_{i-2}(y,z))\\
&=&\phi(z)\phi(f_{i-1}(y,z))-\phi(y)\phi(f_{i-2}(y,z))\\
&=&xu_{i-1}-au_{i-2}=u_i.\\
\end{array}$$
Thus $\phi(f_i(y,z))=u_i$ for all $i\geqslant 1$. In particular,
we have $\phi(f_n(y,z))=u_n$, and hence $\phi((z-y-1)f_n(y,z))=(x-a-1)u_n=0$
by Corollary \ref{2}. It follows that $\phi(I)=0$, and that $\phi$ induces a ring
epimorphism $\overline{\phi}: {\mathbb Z}[y,z]/I\rightarrow r(H_n(q))$
such that $\overline{\phi}(\overline{v})=\phi(v)$ for all $v\in\mathbb{Z}[y,z]$,
where $\overline{v}$ denotes the image of $v$ under the natural epimorphism
$\mathbb{Z}[y,z]\rightarrow \mathbb{Z}[y,z]/I$.

Let $A$ be the subring of $r(H_n(q))$ generated by $a$. $A={\mathbb Z}\langle a\rangle$ is the group ring of the cyclic group $\langle a\rangle$ over $\mathbb Z$. By Corollary \ref{2} we have $u_1=1\in A$ and $u_2=x\in Ax\subset A+Ax$.
By induction on $i$ one can show that $u_i\in A+Ax+\cdots Ax^{i-1}$ for all $i\geqslant 1$.
Hence $u_i\in A+Ax+\cdots Ax^{n-1}$ for all $1\leqslant i\leqslant n$.
Thus for all $1\leqslant i\leqslant n$ and $r\in{\mathbb Z}_n$, by Lemma \ref{1}(1)
we have $[M(i,r)]=a^{n-r}[M(i,0)]=a^{n-r}u_i\in A+Ax+\cdots Ax^{n-1}$.
It follows that $r(H_n(q))=A+Ax+\cdots Ax^{n-1}$. Since $A$ is a free $\mathbb Z$-module
with a $\mathbb Z$-basis $\{a^i|0\leqslant i\leqslant n-1\}$,
$r(H_n(q))$ is generated by elements $a^ix^j, 0\leqslant i, j\leqslant n-1$ as a $\mathbb Z$-module.
Since $r(H_n(q))$ is a free $\mathbb Z$-module of rank $n^2$,
$\{a^ix^j|0\leqslant i, j\leqslant n-1\}$ forms a $\mathbb Z$-basis for $r(H_n(q))$.
Hence one can define a $\mathbb Z$-module homomorphism:
$$\psi: r(H_n(q))\rightarrow {\mathbb Z}[y,z]/I,\ \ a^ix^j\mapsto\overline{y^iz^j}
=\overline{y}^i\overline{z}^j,\ \ 0\leqslant i, j\leqslant n-1.$$
Obviously, ${\mathbb Z}[y,z]/I$ is generated by elements, $\overline{y^iz^j}, 0\leqslant i, j\leqslant n-1$,
as a $\mathbb Z$-module. Now we have
$$\psi\overline{\phi}(\overline{y^iz^j})=\psi\phi(y^iz^j)=\psi(a^ix^j)=\overline{y^iz^j}$$
for all $0\leqslant i, j\leqslant n-1$. Hence $\psi\overline{\phi}={\rm id}$, and so
$\overline{\phi}$ is injective. Thus, $\overline{\phi}$ is a ring isomorphism.
\end{proof}

The coefficients of the generalized Fibonacci polynomial $f_n(y,z)$ can be computed. They are quite similar to those of {\it the standard generalized Fibonacci polynomial} defined by
$$F_1(y,z)=1, \ F_2(y,z)=z,\ \mathrm{and}\
F_n(y,z)=zF_{n-1}(y,z)+yF_{n-2}(y,z),\ n\geqslant 3.$$
For completeness, we compute $f_n(y,z)$ in the following lemma, which might be found elsewhere.

\begin{lemma}
Let ${\mathbb Z}[y,z]$ be the polynomial algebra over $\mathbb Z$ in two variables
$y$ and $z$. Then for any $n\geqslant 1$, we have
\begin{equation}
f_{n}(y, z)=\sum\limits_{i=0}^{[(n-1)/2]}(-1)^i\left[
                                              \begin{array}{c}
                                                n-1-i\\
                                                i\\
                                              \end{array}
                                            \right]y^iz^{n-1-2i}.
\end{equation}
\end{lemma}

\begin{proof}
We prove it by induction on $n$. It is easy to check that Equation (1)  holds
for $1\leqslant n\leqslant 4$. Now let $n>4$ and assume that the equation holds
for smaller positive integers. If $n=2m+1$ is odd, then we have
$$\begin{array}{rcl}
f_{n}(y, z)&=&zf_{2m}(y, z)-yf_{2m-1}\\
&=&\sum\limits_{i=0}^{m-1}(-1)^i\left[\begin{array}{c}
                                      2m-1-i\\
                                      i\\
                                      \end{array}
                                \right]y^iz^{2m-2i}\\
&&-\sum\limits_{i=0}^{m-1}(-1)^i\left[\begin{array}{c}
                                      2m-2-i\\
                                      i\\
                                      \end{array}
                                \right]y^{i+1}z^{2m-2-2i}\\
&=&\sum\limits_{i=0}^{m-1}(-1)^i\left[\begin{array}{c}
                                      2m-1-i\\
                                      i\\
                                      \end{array}
                                \right]y^iz^{2m-2i}\\
&&+\sum\limits_{i=1}^{m}(-1)^i\left[\begin{array}{c}
                                      2m-1-i\\
                                      i-1\\
                                      \end{array}
                                \right]y^{i}z^{2m-2i}\\
&=&z^{2m}+\sum\limits_{i=1}^{m-1}(-1)^i(\left[\begin{array}{c}
                                      2m-1-i\\
                                      i\\
                                      \end{array}
                                \right]+\left[\begin{array}{c}
                                      2m-1-i\\
                                      i-1\\
                                      \end{array}
                                \right])y^{i}z^{2m-2i}\\
&&\ \ \ \ \ +(-1)^{m}y^{m}\\
&=&\sum\limits_{i=0}^{m}(-1)^i\left[\begin{array}{c}
                                      2m-i\\
                                      i\\
                                      \end{array}
                                \right]y^{i}z^{2m-2i}\\
&=&\sum\limits_{i=0}^{[(n-1)/2]}(-1)^i\left[\begin{array}{c}
                                      n-1-i\\
                                      i\\
                                      \end{array}
                                \right]y^{i}z^{n-1-2i}.\\
\end{array}$$
If $n=2(m+1)$ is even, then we have
$$\begin{array}{rcl}
f_{n}&=&zf_{2m+1}(y, z)-yf_{2m}\\
&=&\sum\limits_{i=0}^{m}(-1)^i\left[\begin{array}{c}
                                      2m-i\\
                                      i\\
                                      \end{array}
                                \right]y^iz^{2m+1-2i}\\
&&-\sum\limits_{i=0}^{m-1}(-1)^i\left[\begin{array}{c}
                                      2m-1-i\\
                                      i\\
                                      \end{array}
                                \right]y^{i+1}z^{2m-1-2i}\\
&=&\sum\limits_{i=0}^{m}(-1)^i\left[\begin{array}{c}
                                      2m-i\\
                                      i\\
                                      \end{array}
                                \right]y^iz^{2m+1-2i}\\
&&+\sum\limits_{i=1}^{m}(-1)^i\left[\begin{array}{c}
                                      2m-i\\
                                      i-1\\
                                      \end{array}
                                \right]y^iz^{2m+1-2i}\\
&=&z^{2m+1}+\sum\limits_{i=1}^m(-1)^i(\left[\begin{array}{c}
                                      2m-i\\
                                      i\\
                                      \end{array}
                                \right]+\left[\begin{array}{c}
                                      2m-i\\
                                      i-1\\
                                      \end{array}
                                \right])y^iz^{2m+1-2i}\\
&=&z^{2m+1}+\sum\limits_{i=1}^m(-1)^i\left[\begin{array}{c}
                                      2m+1-i\\
                                      i\\
                                      \end{array}
                                \right]y^iz^{2m+1-2i}\\
&=&\sum\limits_{i=0}^m(-1)^i\left[\begin{array}{c}
                                      2m+1-i\\
                                      i\\
                                      \end{array}
                                \right]y^iz^{2m+1-2i}\\
&=&\sum\limits_{i=0}^{[(n-1)/2]}(-1)^i\left[\begin{array}{c}
                                      n-1-i\\
                                      i\\
                                      \end{array}
                                \right]y^iz^{n-1-2i}.\\
\end{array}$$
Thus the proof is completed.
\end{proof}

Now we can easily derive the Green rings $r(H_n(q))$ for
 $n=2,3, ... , 8$.
\begin{corollary}
When $n=2$, $r(H_2(q))\cong {\mathbb Z}[y,z]/(y^2-1, (z-y-1)z)$.

When $n=3$, $r(H_3(q))\cong {\mathbb Z}[y,z]/(y^3-1, (z-y-1)(z^2-y))$.

When $n=4$, $r(H_4(q))\cong {\mathbb Z}[y,z]/(y^4-1, (z-y-1)(z^3-2yz))$.

When $n=5$, $r(H_5(q))\cong {\mathbb Z}[y,z]/(y^5-1, (z-y-1)(z^4-3yz^2+y^2))$.

When $n=6$, $r(H_6(q))\cong {\mathbb Z}[y,z]/(y^6-1, (z-y-1)(z^5-4yz^3+3y^2z))$.

When $n=7$, $r(H_7(q))\cong {\mathbb Z}[y,z]/(y^7-1, (z-y-1)(z^6-5yz^4+6y^2z^2-y^3))$.

When $n=8$, $r(H_8(q))\cong {\mathbb Z}[y,z]/(y^8-1, (z-y-1)(z^7-6yz^5+10y^2z^3-4y^3z))$. \hfill$\Box$
\end{corollary}
\vspace{1cm}

\begin{remark}
\begin{enumerate}
\item One can easily see that the Grothendick ring of $H_n(q)$ is the group ring $k\mathbb{Z}_n$ generated by the simple module $M(1,0)$. From the above examples, we see that the Green ring is much more complicated than the Grothendick ring.
\item The Green rings of generalized Taft algebras and the Green rings of monomial Hopf algebras
\cite{chyz} can be computed in a similar way. However, the computation of the Green ring of the small quantum group or the Green ring of the quantum double of a Taft algebra seem to be much more complicated as they are not finitely generated \cite{Ch3,ks}.
\item Since the module category of a quasitriangular Hopf algebra $H$ is braided monoidal, the Green ring of $H$ is commutative. The Taft algebra $H_n(q)$ is not quasitriangular in case $n> 2$ (even not almost cocommutative, see \cite{Cib}), but its Green ring is commutative. This leads to the following question: can we characterize the class of Hopf algebras whose Green ring is commutative?
\end{enumerate}
\end{remark}

\section*{ACKNOWLEDGMENTS}

The authors are grateful to the referee for suggesting a formula for the generalized Fibonacci polynomials $f_n(y,z)$. The first named  author would like to thank the Department of Mathematics, University of Antwerp
for its hospitality during his visit in 2011. He is grateful to the Belgium FWO for financial support.
He was also supported by NSF of China (No. 11171291).\\


\begin{thebibliography}{99}

\bibitem{Archer}
L. Archer, On certain quotients of the Green rings of dihedral 2-groups,
J. Pure \& Appl. Algebra 212 (2008), 1888-1897.
\bibitem{BenCar}
D. J. Benson and J. F. Carlson, Nilpotent elements in the Green ring,
J. Algebra 104 (1986), 329-350.
\bibitem{BenPar}
D. J. Benson and R. A. Parker, The Green ring of a finite group, J. Algebra 87 (1984), 290-331.
\bibitem{BrJoh}
R. M. Bryant and M. Johnson, Periodicity of Adams operations on the Green
ring of a finite group, arXiv:0912.2933v1[math.RT].
\bibitem{Ch3}
H. X. Chen, Finite-dimensional
representations of a quantum double, J. Algebra 251 (2002),
751-789.
\bibitem{chyz}
X. Chen, H. Huang, Y. Ye and P. Zhang, Monomial Hopf algebras. J. Algebra
275 (2004), 212-232.
\bibitem{Cib}
C. Cibils, A quiver quantum group, Commun. Math. Phys. 157 (1993), 459-477.
\bibitem{Chin}
W. Chin, Special biserial coalgebras and representations of quantum SL(2), arXiv:math/0609370v5 [math.QA].
\bibitem{Green}
J. A. Green, The modular representation algebra of a finite group, Ill. J.
Math. 6(4) (1962), 607-619.
\bibitem{Gunn}
E. Gunnlaugsd${\rm{\acute{o}}}$ttir, Monoidal structure of the category of $\mathfrak{u}^+_q$-modules,
Linear Algebra and its Applications 365 (2003), 183-199.
\bibitem{HTW}
I. Hambleton, L. R. Taylor and E.B. Williams, Dress induction and Burnside
quotient Green ring, arXiv:0803.3931v2[math.GR]
\bibitem{Ka}
C. Kassel, Quantum groups, Springer-Verlag, New York, 1995.
\bibitem{ks}
H. Kondo and Y. Saito, Indecomposable decomposition of tensor products of modules over the restricted quantum universal enveloping algebra associated to $\mathfrak{sl}_2$,  J. Alg. 330(2011), 103-129.
\bibitem{Maj}
S. Majid, Foundations of quantum group theory, Cambridge Univ. Press,  Cambridge, 1995.
\bibitem{Mon}
S. Montgomery,  Hopf Algebras and their actions on rings,  CBMS Series in Math.,  Vol.
82, Am. Math. Soc., Providence, 1993.
\bibitem{Sw}
M. E. Sweedler, Hopf Algebras, Benjamin, New York, 1969.
\bibitem{Ta}
E. J. Taft, The order of the antipode of a finite-dimensional Hopf algebra, Proc. Nat. Acad. Sci. USA 68 (1971), 2631-2633.
\bibitem{With}
S. J. Witherspoon, The representation ring of the quantum double of a finite group, J. Algebra 179 (1996), 305-329.
\end{thebibliography}
\end{document}